\documentclass[12pt]{amsart}
\usepackage{amsfonts}
\usepackage{amssymb}
 \newtheorem{theorem}{Theorem}[section]

 \newtheorem{lemma}[theorem]{Lemma}
 
 \theoremstyle{definition}
 
 \theoremstyle{remark}

\numberwithin{equation}{section}

\begin{document}


\baselineskip=17pt


\title[Additive bases of $C_3\oplus C_{3q}$]{Additive bases of $C_3\oplus C_{3q}$}

\author[Y.K. Qu]{Yongke Qu}
\address{Department of Mathematics\\ Luoyang Normal University\\
LuoYang 471934\\ P.R. China}
\email{yongke1239@163.com}

\author[Y.L. Li]{Yuanlin Li*}
\address{Department of Mathematics and Statistics\\  Brock University\\
St. Catharines, ON L2S 3A1, Canada}
\email{yli@brock.ca}

\thanks{*Corresponding author: Yuanlin Li, E-mail: yli@brocku.ca}

\date{}

\begin{abstract}
Let $G$ be a finite abelian group and $p$ be the smallest prime dividing $|G|$. Let $S$ be a sequence over $G$.
We say that $S$ is regular if for every proper subgroup $H \subsetneq G$, $S$ contains at most $|H|-1$ terms from
$H$. Let $\mathsf c_0(G)$ be the smallest integer $t$ such that every regular sequence $S$ over $G$ of length
$|S|\geq t$ forms an additive basis of $G$, i.e., $\sum(S)=G$. The invariant  $\mathsf c_0(G)$ was first studied
by Olson and Peng in 1980's, and since then it has been determined for all finite abelian groups except for the
groups with rank 2 and a few groups of rank 3 or 4 with order less than $10^8$. In this paper, we focus on the
remaining case concerning groups of rank 2. It was conjectured by the first author and Han (Int. J. Number Theory 13 (2017)
    2453-2459) that $\mathsf c_0(G)=pn+2p-3$ where
$G=C_p\oplus C_{pn}$ with $n\geq 3$. We confirm the conjecture for the case when $p=3$ and $n=q \,(\geq 5)$ is a
prime number.
\end{abstract}



\maketitle

\section{Introduction and main results}
Through the paper, let $G$ be a finite abelian group, written additively, $p$ be the smallest prime dividing $|G|$
and $\mathsf r(G)$ denote the rank of $G$. Let $S$ be a sequence over $G$. We say that $S$ is an {\sl additive
basis} of $G$ if every element of $G$ can be expressed as the sum over a nonempty subsequence of $S$. For every
subgroup $H$ of $G$, let $S_H$ denote the  subsequence of $S$ consisting of all terms of $S$ contained in $H$. We
say that $S$ is a regular sequence over $G$ if $|S_H|\leq |H|-1$ holds for every subgroup $H \subsetneq G$. Let
$\mathsf c_0(G)$ denote the smallest integer $t$ such that every regular sequence over $G$ of length at least $t$
is an additive basis of $G$. The problem of determining $\mathsf c_0(G)$ was first proposed by Olson and it was
conjectured that $\mathsf c_0(C_p\oplus C_p)=2p-1$. In 1987, Peng proved this conjecture and further determined
$\mathsf c_0(G)$ for all the finite elementary abelian $p$-groups (\cite{Peng1,Peng2}). Recently, the problem
related to the additive basis of a finite abelian group has been investigated by several authors (\cite{GHQQZ, GQZ,
QH1, QH2}). In particular, $\mathsf c_0(G)$ has been determined for any of the following finite abelian groups:
\begin{enumerate}
\item $G$ is cyclic;

\item $|G|$ is even;

\item $\mathsf r(G)\geq 4$ and $G\neq C_3^3\oplus C_{3n}$ where $n>3$ is odd and is not a power of $3$ with $|G|<
    3.72\times 10^7$ ;

\item $\mathsf r(G)=3$ and either $p\geq 11$ or $3\leq p \leq 7$ with $|G| \geq 3.72\times 10^7$;

\item $\mathsf r(G)\geq 2$ and $G$ is a $p$-group.

\end{enumerate}

 In this paper, we focus our investigation on the remaining case when $G$ is of rank $2$. It was conjectured in
 \cite{QH2} that $\mathsf c_0(G)=pn+2p-3$ where $G=C_p\oplus C_{pn}$ with $n\geq 3$. We remark that the existing
 methods used to compute this invariant for groups of rank greater than $2$ cannot be applied directly to calculate
 $\mathsf c_0(G)$ when $G$ is a group of rank 2. We adopt a new method (i.e., use group algebras as a tool) and we
 are able to confirm the above conjecture for the case when  $p=3$ and $n=q \,(\geq 5)$ is a prime.

\begin{theorem} \label{mainthm} Let $G=C_3\oplus C_{3q}$ be a finite abelian group with a prime $q\geq 5$. Then
$\mathsf c_0(G)=3q+3$.
\end{theorem}

\section{Notations and Preliminaries}
Suppose that $G_0\subseteq G$ is a subset of $G$ and $\mathcal{F}(G_0)$ is the multiplicatively written, free
abelian monoid with basis $G_0$. The elements of $\mathcal{F}(G_0)$ are called {\it sequences} over
$G_0$. A sequence $S$ over $G_0$ will be written in the form
$$
S = g_1 \cdot \ldots \cdot g_{\ell}=\Pi_{i\in [1,\ell]}g_i, $$
where $g_i \in G_0$ for all $1\leq i\leq {\ell}$. We say $T=\Pi_{i\in I}g_i$ a subsequence of $S$ and denote by
$T|S$, where $I\subseteq [1, \ell]$. For a subsequence $T|S$, let $I_T=\{i\in [1, \ell]~|~g_i|T\}$. We set $T=1$ if
$I_T=\emptyset$. We call
\begin{eqnarray*}
|S|&=&\ell \in \mathbb N_0 \quad  \text{the {\it length} of } \ S,\\
\sigma (S)&=&\sum_{i=1}^{\ell}g_i\in G\quad  \text{the {\it sum} of} \ S.
\end{eqnarray*}
Define
  $$
  \sum(S)=\{\sigma(T): \ 1 \neq  T\mid S\},
  $$
\noindent and $$
  \sum\nolimits_0(S)=\sum(S)\cup\{0\}.
  $$
\noindent We call a sequence $S$ a zero-sum sequence if $\sigma(S)=0$, and a zero-sumfree sequence if $0\notin
\sum(S)$.

Let $\mathsf D(G)$ denote the Davenport constant of $G$, which is defined as the smallest integer $t$ such that
every sequence $S$ over $G$ of length $|S|\geq t$ contains a nonempty zero-sum subsequence. Let $\mathsf d(G)$
denote the maximal length of a zero-sumfree sequence over $G$. Then $\mathsf d(G)=\mathsf D(G)-1$. The exact value
of $\mathsf D(G)$ has been determined only for a few classes of groups, such as finite abelian $p$-groups, abelian
groups rank not exceeding $2$, and certain very special abelian groups of rank $3$ (\cite{GL}). Notice that
$\mathsf D(C_{n_1}\oplus C_{n_2})=n_1+n_2-1$, where $1\leq n_1|n_2$ (\cite[Theorem 5.8.5]{GeK}).

For each subset $A$ of $G$, denote by $\langle A \rangle$ the subgroup generated by $A$.  Let ${\rm st}(A)=\{g\in
G: g+A=A\}$. Then ${\rm st}(A)$ is the maximal subgroup $H$ of $G$ such that $H+A=A$. The following is the well
known Kneser's theorem and a proof of it can be found in \cite{Na}.

\begin{lemma}(Kneser) \cite[Theorem 4.4]{Na}\label{Kneser} Let $A_1, \ldots, A_r$ be nonempty finite subsets of an
abelian group $G$, and let $H={\rm st}(A_1+\cdots +A_r)$. Then,
$$
|A_1+\cdots +A_r|\geq |A_1+H|+\cdots +|A_r+H|-(r-1)|H|.
$$
\end{lemma}

We note that if $H=G$ then $A_1+\cdots +A_r=G$. So when using the above Kneser's Theorem to prove $A_1+\cdots
+A_r=G$, we need only consider the case when $H\neq G$.

The following three lemmas provide some results concerning the additive basis and its inverse problem, which will be
needed in sequel.

\begin{lemma}\cite[Theorem 1.1]{QH1}\label{cyclic}
Let $G$ be a cyclic group with order $n$, and let $S$ be a regular sequence of length $|S|=n-1$ over $G$. If
$\sum(S)\neq G$, then $S=g^{n-1}$ where $g$ is a generator of $G$.
\end{lemma}

\begin{lemma}\cite[Lemma 2.3(1)]{GHQQZ}\label{st=0}
Let $G$ be a finite abelian group, and let $p$ be the smallest prime dividing $|G|$. Let $S$ be a regular sequence over $G$ of length $|S|\geq \mbox{max}\{\frac{|G|}{p}+p-2, \mathsf D(G)\}$. If $\sum(S)\neq G$, then $\mbox{st}(\sum(S))=\{0\}$.
\end{lemma}

\begin{lemma}\cite[Theorem 2]{Peng1}\cite[Lemma 3.12]{GPZ2013}\label{rank2}
Let $G=C_p\oplus C_p$ and $S$ be a regular sequence over $G$. Then $\mathsf c_0(G)=2p-1$. Moreover, if $|S|=2p-2$,
then $|\sum_0(S)|\geq |G|-1$.
\end{lemma}

For a field $\mathbf{F}$, let $\mathbf{F}G$ denote the group algebra of $G$ over $\mathbf{F}$ and $d(G, \mathbf{F})$ be the largest integer $\ell\in \mathbb{N}$ having the following
property:

There exists some sequence $S=g_1\cdot \ldots \cdot g_{\ell}$ over $G$ of length $\ell$ such that
$$(X^{g_1}-a_1)\cdot \ldots \cdot (X^{g_{\ell}}-a_{\ell})\neq 0 \in \mathbf{F}G \mbox{ for all }
a_1,\ldots,a_{\ell}\in \mathbf{F}^{\times}.$$

\begin{lemma}\cite[Theorem 3.3]{GL}\cite[Theorem 1.1]{S}\label{d(G,F)}
Let $G$ be a finite abelian group and $\mathbf{F}$ be a splitting field of $G$. Then,
\begin{enumerate}

\item if $G=C_2\oplus C_{2n}$, then $d(G, \mathbf{F})=d(G)=2n$;

\item if $G=C_3\oplus C_{3n}$, then $d(G, \mathbf{F})=d(G)=3n+1$.
\end{enumerate}
\end{lemma}

For any $\alpha\in \mathbf{F}G$, by $L_{\alpha}$ we denote the set of elements $g\in G$ such that $\alpha
(X^g-a)=0$ holds for some $a\in \mathbf{F}^{\times}$. We note that the statement in the following lemma is slightly
more general than that in \cite[Lemma 5]{QH2}; however, the same proof carries over.

\begin{lemma}\cite[Lemma 5]{QH2}\label{coset}
Let $G$ be a finite abelian group and $S=g_1\cdot \ldots \cdot g_\ell$ be a sequence over $G$. Suppose that
$\alpha=(X^{g_1}-a_1)\cdot \ldots \cdot (X^{g_{\ell}}-a_{\ell})\neq0$ for some $a_1,\ldots,a_{\ell}\in
\mathbf{F}^{\times}$ and $H=L_{\alpha}$, then $\sum(S)\supseteq (g_0+H)\setminus\{0\}$ for some $g_0\in G$.
Moreover, $\sum(Sh)\supseteq g_0+H$ for any $h\in H$.
\end{lemma}

\section{Proof of Theorem \ref{mainthm}}
Let $G=C_3\oplus C_{3q}=H\oplus K$, where $H\cong C_3\oplus C_3$, $K\cong C_q$ and $q \geq 5$ is a prime. Let
$S={0\choose1}^{3q-2}{1\choose-1}^{4}$ be a sequence over $G$ with length $3q+2$. Then $S$ is regular and
${2\choose 3q-3}\notin\sum(S)$. It follows from this example that $\mathsf c_0(C_3\oplus C_{3q})\geq 3q+3$. To show
the equality holds it is sufficient to prove every regular sequence over $G$ with length $3q+3$ forms an additive
basis of $G$. Let $S=g_1\cdot \ldots \cdot g_\ell$ be a regular sequence over $G$ with length $\ell=3q+3$,
$S_H=g_1\cdot \ldots \cdot g_t$ and  $S_K=g_{t+1}\cdot \ldots \cdot g_{t+r}$. We first prove the following crucial
lemma, which provides some sufficient conditions for a regular sequence $S$ of length $3q+3$ to be an additive basis.

\begin{lemma}\label{basic}
Let $G=C_3\oplus C_{3q}$, and $S$ be a regular sequence over $G$ with length $\ell=3q+3$, where $q \geq 5$ is a
prime. Then $\sum(S)=G$ if any of the following conditions holds:
\begin{itemize}
\item[(i)] There exists a nontrivial subgroup $H'$ of $G$ such that $\sum_0(S_{H'})=H'$;

\item[(ii)] There exist a subsequence $S'|S$ and a nontrivial subgroup $H'$ of $G$ such that $\sum_0(S')\supseteq
    a+H'$ for some $a\in G$ and $(\ell-|I_{S_M}\cup I_{S'}|+1)|M|\geq 9q $ where
    $M=st((a+H')+\sum_0(SS'^{-1}))$;

\item[(iii)] There exist some $b_i\in \mathbf{F}^{\times}$ for all $i\in I_{S_H}\cup I_{S_K}$ such that
    $\Pi_{i\in I_{S_{H}}\cup I_{S_K}}(X^{g_i}-b_i)=0$.
\end{itemize}
\end{lemma}
\begin{proof}
(i) Since $\sum_0(S_{H'})=H'$ for some nontrivial subgroup $H'$ of $G$ by assumption, we have $\sum(S)+H'= \sum(S)$,  so $st(\sum(S))\supseteq H'\neq \{0\}$. Since $S$ is a regular sequence and $|S|\geq \mbox{max}\{\frac{|G|}{3}+3-2=3q+1, \mathsf D(G)=3q+2\}$, by Lemma~\ref{st=0} we have $\sum(S)=G$.

(ii) Let $A=(a+H')+\sum_0(SS'^{-1})$. By assumption, we have $\sum(S)\supseteq A$ and $M\supseteq H'$. By
Lemma~\ref{Kneser}, we have
\begin{align*}
|\sum(S)|&\geq |a+H'+M|+\sum_{i\in I_S\setminus (I_{S_M}\cup I_{S'})}|\{0, g_i\}+M|-(\ell-|I_{S_M}\cup
I_{S'}|)|M|\\
&\geq (\ell-|I_{S_M}\cup I_{S'}|+1)|M|\\
&\geq |G|.
\end{align*}
Thus $\sum(S)=G$.

(iii) Recall that $S_H=g_1\cdot \ldots \cdot g_t$ and $S_K=g_{t+1}\cdot \ldots \cdot g_{t+r}$. Since $\Pi_{i\in
[1,t+r]}(X^{g_i}-b_i)=0$ for some $b_i\in \mathbf{F}^{\times}$ and $G=H\oplus K$, we conclude that either
$\Pi_{i\in [1,t]}(X^{g_i}-b_i)=0$ or $\Pi_{i\in [t+1,t+r]}(X^{g_i}-b_i)=0$. Otherwise, if both  $\Pi_{i\in
[1,t]}(X^{g_i}-b_i)\neq 0$ and $\Pi_{i\in [t+1,t+r]}(X^{g_i}-b_i)\neq 0$ hold, then let
$$\Pi_{i\in [1,t]}(X^{g_i}-b_i)=\sum_{h\in H}c_hX^h, \ \ \ \ \Pi_{i\in [t+1,t+r]}(X^{g_i}-b_i)=\sum_{k\in K}c_kX^k,$$
and $$\Pi_{i\in I_{S_{H}}\cup I_{S_K}}(X^{g_i}-b_i)=\sum_{g\in G}e_gX^g.$$
We have $\Pi_{i\in I_{S_{H}}\cup I_{S_K}}(X^{g_i}-b_i)=(\sum_{h\in H}c_hX^h)(\sum_{k\in K}c_kX^k)=\sum_{h\in H, k\in
K}c_hc_kX^{h+k}$, and thus $e_g=\sum_{g=h+k, h\in H, k\in K}c_hc_k$. Since $G=H\oplus K$, for each $g\in G$,
$e_g=c_hc_k$ for unique $h\in H$, $k\in K$ with $g=h+k$. Since $\Pi_{i\in [1,t]}(X^{g_i}-b_i)\neq 0$ and $\Pi_{i\in
[t+1,t+r]}(X^{g_i}-b_i)\neq 0$, we have $c_{h_0}\neq 0$ and $c_{k_0}\neq 0$ for some $h_0\in H$ and $k_0\in K$.
Therefore, $e_{h_0+k_0}=c_{h_0}c_{k_0}\neq 0$, yielding a contradiction to the assumption.

If $\Pi_{i\in [t+1,t+r]}(X^{g_i}-b_i)=0$, then by Lemma~\ref{coset}, $\sum_0(S_K)=K$. By~(i), $\sum(S)=G$.

Next assume that $\Pi_{i\in [1,t]}(X^{g_i}-b_i)=0$. Since $S$ is regular, $|S_{H}|\leq |H|-1=8$. If $|S_{H}|\geq
5$, then by Lemma~\ref{rank2} $\sum(S_{H})=H$. By (i), $\sum(S)=G$.

We now consider the case when $|S_{H}|\leq 4$. Since $\Pi_{i\in [1,t]}(X^{g_i}-b_i)=0$, by Lemma~\ref{coset},
$\sum_0(S_{H})\supseteq a+N$ for some nontrivial subgroup $N\subseteq H$. Let $A=(a+N)+\sum_0(SS_{H}^{-1})$ and
$st(A)=M$. Then $\sum(S)\supseteq A$ and $M\supseteq N$. If $M=N$, then $M\subseteq H$. Thus $S_M|S_{H}$.
Therefore, $(\ell-|I_{S_M}\cup I_{S_H}|+1)|M|\geq(\ell-3)|M|\geq 9q$. By~(ii), $\sum(S)=G$.

Suppose that $M\supsetneq N$. If $M\supseteq H$, then $S_{H}|S_M$. Since $|S_M|\leq |M|-1$, we have
$(\ell-|I_{S_M}\cup I_{S_H}|+1)|M|=(\ell-|I_{S_M}|+1)|M| \geq(3q-|M|+5)|M|\geq 9q$. By~(ii), $\sum(S)=G$.

If $M\nsupseteq H$, then $M\cong C_{3q}$. If $|S_M|\geq |M|-1$, then by Lemma~\ref{cyclic} $\sum_0(S_M)=M$. Thus
by~(i), $\sum(S)=G$ and we are done. So we may assume that $|S_M|\leq |M|-2$. If there exists a subgroup $N_0$ of
order $3$ such that $|S_{N_0}|\geq 2$, then $|S_{N_0}|=2$ and $\sum_0(S_{N_0})=N_0$. By~(i), $\sum(S)=G$. Next we
may always assume $|S_{N_0}|\leq 1$ for every subgroup $N_0$ of order $3$. Then $|I_{S_H}\setminus I_{S_{H'}}|\leq
3$ where $H'$ is a cyclic subgroup with order $3$ or $3q$. In particular, $|I_{S_{H}}\setminus I_{S_M}|\leq 3$.
Thus $(\ell-|I_{S_M}\cup I_{S_{H}}|+1)|M| \geq(\ell-|I_{S_M}|-2)|M| \geq(3q-|M|+3)|M|\geq 9q$. Hence by~(ii), again
we obtain $\sum(S)=G$. This completes the proof. \end{proof}

We are now in position to give a proof for the main result.\\

\noindent {\bf Proof of Theorem \ref{mainthm}.}\\

Let $S$ be a regular sequence over $G$ with length $3q+3$. We need to show that $\sum(S)=G$. Assume to the contrary
that $\sum(S)\neq G$. By Lemma~\ref{basic}~(iii), we can find a subsequence $S_1|S$ with maximal length such that
$S_HS_K|S_1$ and $\Pi_{i\in I_{S_1}}(X^{g_i}-a_i)\neq 0$ for all $a_i\in \mathbf{F}^{\times}$ where $i\in I_{S_1}$ and $\mathbf{F}$ is a splitting field of $G$.
Notice that every element of $SS_1^{-1}$ has order $3q$. Without loss of generality, we may assume that
$S_1=\Pi_{i=1}^{m}g_i$ where $m\in[t+r, \ell]$. We distinguish the proof into the following two cases:\\

\noindent {\bf Case 1.} $m\geq 3q+1$.\\

By Lemma~\ref{d(G,F)}, $\mathsf d(G,\mathbf{F})=3q+1$, so $m=3q+1$. Then for any $g\in G$, there exists a subsequence $S_g|S_1$ with minimal length such that $(X^g-b_g)\Pi_{i\in I_{S_g}}(X^{g_i}-b_i)=0$ for some $b_g, b_i\in \mathbf{F}^{\times}$ where $i\in I_{S_g}$. By Lemma \ref{coset}, $\sum_0(S_1)$ contains a coset of $\langle
g\rangle$. Since there are exactly $4$ distinct subgroups of $G$ of order $3q$, and both $g_{3q+2}$ and $ g_{3q+3}$
have order of $3q$,  we can find a subgroup  $\langle g_0\rangle$ of order $3q$ such that $g_{3q+2}, g_{3q+3}\notin
\langle g_0\rangle$. Thus, $(\ell-|I_{S_{\langle g_0\rangle}}\cup I_{S_1}|+1)|M|\geq 3|M|\geq 9q$. By
Lemma~\ref{basic}~(ii), $\sum(S)=G$, yielding a contradiction.\\

\noindent {\bf Case 2.} $m\leq 3q$.\\

For any $g|SS_1^{-1}$, there exists a subsequence $S_g|S_1$ with minimal length such that $S_HS_K|S_g$ and
$(X^{g}-b_{g})\Pi_{i\in I_{S_g}}(X^{g_i}-b_i)=0$ for some $b_g, b_{i}\in \mathbf{F}^{\times}$ where $i\in I_{S_g}$.
Since both $\Pi_{i\in I_{S_g}}(X^{g_i}-b_i)\neq 0$ and $(X^g-b_g)\Pi_{i\in I_{S_g}\setminus\{j\}}(X^{g_i}-b_i)\neq
0$ where $j\notin I_{S_H}\cup I_{S_K}$, it follows from Lemma~\ref{coset} that
\begin{center}
$\sum_0(S_g)$ contains a complete coset of $\langle g\rangle$  \ \ \ \ \    (*)
\end{center}
and
\begin{center}
$\sum_0(S_gg)$ contains a complete coset of $\langle g_j\rangle$ for any $j\in I_{S_g}\setminus (I_{S_H}\cup
I_{S_K})$. \ \ \ (**)
\end{center}
If $\langle g_{3q+1}, g_{3q+2}, g_{3q+3}\rangle$ is not a cyclic group, then without loss of generality, we may
assume that both $g_{3q+2}, g_{3q+3}$ are not in $\langle g_{3q+1}\rangle$. This together with (*), proves that
$(\ell-|I_{S_{\langle g_{3q+1}\rangle}}\cup I_{S_1}|+1)|M|\geq 3|M|\geq 9q$. By Lemma~\ref{basic} (ii),
$\sum(S)=G$,  yielding a contradiction.

Next we assume that $\langle g_{3q+1}, g_{3q+2}, g_{3q+3}\rangle=H_1$ is a cyclic group.  We distinguish the rest
of the proof into the following two subcases:\\

\noindent {\bf Subcase 2.1.} Every element of $S_{g_{3q+1}}g_{3q+1}(S_HS_K)^{-1}$ is contained in the subgroup $H_1$.\\

Let $T=S_{g_{3q+1}}$. Then $I_T\setminus I_{S_{H_1}} \subseteq I_{S_H}$. By (*), $\sum_0(T)$ contains a complete coset
of $\langle g_{3q+1}\rangle$. If $|S_{H_1}|\geq 3q-1$, by Lemma \ref{cyclic}, $\sum_0(S_{H_1})=H_1$. By Lemma
\ref{basic} (i), $\sum(S)=G$, yielding a contradiction.

If $|S_{H_1}|\leq 3q-2$, since $\sum(S)\neq G$, as in the proof of lemma~\ref{basic}~(iii), we obtain
$|I_{S_H}\setminus I_{S_{H_1}}|\leq 3$. Since $I_T\setminus I_{S_{H_1}} \subseteq I_{S_H}$, we have
$(\ell-|I_{S_{H_1}}\cup I_T|+1)|M|\geq (\ell-|I_{S_{H_1}}|-2)|M|\geq 3|M|\geq 9q$. By Lemma \ref{basic} (ii),
$\sum(S)=G$, yielding a contradiction.\\

\noindent {\bf Subcase 2.2.} There exists an element $g_j$ of $S_{g_{3q+1}}g_{3q+1}(S_HS_K)^{-1}$ such that $g_j\notin H_1$ for some $j\in I_{S_{g_{3q+1}}}\setminus (I_{S_H}\cup I_{S_K})$. \\

By (**), $\sum_0(S_{g_{3q+1}}g_{3q+1})$ contains a
complete coset of $\langle g_j\rangle$. Since $\langle g_{3q+2}\rangle=\langle g_{3q+3}\rangle=H_1$, we have
$g_{3q+2}, g_{3q+3}\notin \langle g_j\rangle$. Then $(\ell-|I_{S_{\langle g_j\rangle}}\cup I_{S_1}|+1)|M|\geq
3|M|\geq 9q$. By Lemma \ref{basic} (ii), $\sum(S)=G$, yielding a contradiction.

In all cases we have found contradictions. Thus $\sum(S)=G$ holds as desired.\\

\subsection*{Acknowledgements}
This work was carried out during a visit by the first author to Brock University as an international visiting scholar. He would like to sincerely thank the host institution for its hospitality and for providing an excellent atmosphere for research. This work was supported in part by the National Science Foundation of China (Grant No. 11701256, 11871258), the Youth Backbone Teacher Foundation of Henan's University (Grant No. 2019GGJS196), the China Scholarship Council (Grant No. 201908410132), and it was also supported in part by a Discovery Grant from the Natural Sciences and Engineering Research Council of Canada (Grant No. RGPIN 2017-03903).



\normalsize


\begin{thebibliography}{[HD82]}




\normalsize
\baselineskip=17pt


\bibitem{GHQQZ} W. Gao, D. Han, G. Qian, Y. Qu and H. Zhang, \emph{On additive bases II}, Acta Arith. 168 (2015)
    247--267.

\bibitem{GL} W. Gao and Y. Li, \emph{Remarks on group rings and the Davenport constant}, Ars Combin. 101 (2011)
    417--423.

\bibitem{GPZ2013} W. Gao, J. Peng and Q. Zhong, \emph{A quativative aspect of non-unique factorizations: the
    Narkiewicz constants III}, Acta Arith. 158 (2013), 271--285.

\bibitem{GQZ} W. Gao, Y. Qu and H. Zhang, \emph{On additive bases III}, Acta Arith. 193 (2020) 293--308.

\bibitem{GeK} A. Geroldinger and F. Halter-Koch, \emph{Non-Unique Factorizations. Algebraic, Combinatorial and Analytic Theory}, Pure and Applied Mathematics, vol. 278, Chapman \& Hall/CRC, Boca Raton, 2006.

\bibitem{Na} M. Nathanson, \emph{Additive Number Theory : Inverse Problems and the Geometry of Sumsets}, Graduate Texts in Mathematics, vol. 165, Springer-Verlag, New York, 1996.

\bibitem{Peng1} C. Peng, \emph{Addition theorems in elementary abelian groups I}, J.~Number Theory 27 (1987) 46--57.

\bibitem{Peng2} C. Peng, \emph{Addition theorems in elementary abelian groups II}, J.~Number Theory 27 (1987) 58--62.

\bibitem{QH1} Y. Qu and D. Han, \emph{An inverse theorem for additive bases}, Int. J.~Number Theory 12 (2016)
    1509--1518.

\bibitem{QH2} Y. Qu and D. Han, \emph{Additive bases of $C_p\oplus C_{p^n}$}, Int. J.~Number Theory 13 (2017)
    2453--2459.

\bibitem{S} D. Smertnig, \emph{On the Davenport constant and group algebras}, Colloq. Math. 121 (2010) 179--193.

\end{thebibliography}
\end{document}